\documentclass[12pt,leqno]{article}
 \usepackage{fullpage}
\usepackage{amssymb}
\usepackage{amsmath}
\usepackage{amsthm}
\newtheorem{theorem}{Theorem}

\newtheorem{conjecture}{Conjecture}

\newtheorem{lemma}{Lemma}
\newtheorem{remark}{Remark}
\begin{document}

\title{\large\bf Proof of a conjecture related to divisibility properties of binomial coefficients\footnote{This work was supported by the National
Natural Science Foundation of China, Grant No. 11371195. }}
\date{}
\author{Quan-Hui Yang\footnote{Email: yangquanhui01@163.com.}\\
\small School of Mathematics and Statistics,
\\ \small Nanjing University of Information Science and Technology,\\
\small  Nanjing 210044, P. R. China}
 \maketitle
\vskip 3mm

\begin{abstract} Let $a,b$ and $n$ be positive integers with $a>b$. In this note,
we prove that
$$(2bn+1)(2bn+3){2bn \choose bn}\bigg|3(a-b)(3a-b){2an \choose an}{an\choose bn}.$$ This
confirms a recent conjecture of Amdeberhan and Moll.


\end{abstract}

\section{Introduction}
Let $\mathbb{Z}$ denote the set of all integers. In 2009, Bober \cite{bober} determined all
cases such that
$$\frac{(a_1n)!\cdots (a_kn)!}{(b_1n)!\cdots (b_{k+1}n)!}\in \mathbb{Z},$$
where $a_s\not=b_t$ for all $s,t$, $\sum{a_s}=\sum{b_t}$ and $\gcd(a_1,\ldots,a_k,b_1,\ldots,b_{k+1})=1$.

Recently, Z.-W. Sun \cite{sun12,sun13} studied divisibility properties of binomial
coefficients and obtained some interesting results. For example,
$$2(2n+1){2n \choose n}\bigg| {6n \choose 3n}{3n \choose n},$$
$$(10n+1){3n \choose n}\bigg| {15n \choose 5n}{5n-1 \choose n-1}.$$
Later, Guo and Krattenthaler (see \cite{guo,{guokra}}) got some similar divisibility results.
For related results, one can refer to \cite{Calkin}-\cite{Fine} and \cite{Guo07}-\cite{Razpet}.

Let $$S_n=\frac{{6n \choose 3n} {3n \choose n}}{2(2n+1){2n \choose n}}\quad \text{and} \quad
t_n=\frac{{15n \choose 5n}{5n-1 \choose n-1}}{(10n+1){3n \choose n}}.$$

In \cite{guonew}, Guo proved the following Sun's conjectures.

\noindent{\bf  Theorem A.} (See \cite[Conjecture 3(i)]{sun12}.){\em~ Let $n$ be a positive integer. Then
$$3S_n\equiv 0 \pmod {2n+3}.$$}
\noindent{\bf  Theorem B.} (See \cite[Conjecture 1.3]{sun13}.){\em ~ Let $n$ be a positive integer. Then
$$21t_n\equiv 0 \pmod {10n+3}.$$}
Some other results are also proved.

T. Amdeberhan and  V. H. Moll proposed the following
conjecture  (See Guo \cite[Conjecture 7.1]{guonew}).

\begin{conjecture}\label{conj} Let $a,b$ and $n$ be positive integers with $a>b$. Then
$$(2bn+1)(2bn+3){2bn \choose bn}\bigg|3(a-b)(3a-b){2an \choose an}{an\choose bn}.$$
\end{conjecture}

In this note, we give the proof of this conjecture.

\begin{theorem}\label{thm1} Conjecture \ref{conj} is true.
\end{theorem}
\begin{remark} Let $a=3$ and $b=1$. Then Theorem A follows from Theorem \ref{thm1} immediately.
\end{remark}

\section{Proofs}

For an integer $n$ and a prime $p$, we write $p^{k}\|n$ if $p^k|n$ and $p^{k+1}\nmid n$.
We use $\nu_p(n)$ to denote such integer $k$. It is well known that
\begin{equation}\label{eq1}\nu_p{(n!)}=\sum_{i=1}^{\infty}\left\lfloor \frac {n}{p^i}\right\rfloor,
\end{equation}
where $\left\lfloor x\right\rfloor$ denotes the greatest integer not exceeding $x$.

Before the proof of Theorem \ref{thm1}, we give the following lemma.

\begin{lemma}\label{lem1} Let $x$ and $y$ be two real numbers. Then
$\left\lfloor 2x\right\rfloor +\left\lfloor y\right\rfloor \ge \left\lfloor x\right\rfloor +\left\lfloor x-y\right\rfloor +\left\lfloor 2y\right\rfloor.$
\end{lemma}
\begin{proof} Noting that $2x+y=x+(x-y)+2y$, we only need to prove that
$\{2x\}+\{y\}\le \{x\}+\{x-y\}+\{2y\}$, where $\{z\}$ denotes the fractional part of $z$.
We can prove it by comparing $\{x\}$ and $\{y\}$ with 1/2. We leave the proof to the reader.
\end{proof}

\begin{proof}[Proof of Theorem \ref{thm1}] Let
$$T(a,b,n):={2an \choose an}{an \choose bn}\bigg/{2bn \choose bn}=\frac{(2an)!(bn)!}{(an)!(an-bn)!(2bn)!}.$$
By \eqref{eq1}, for any prime $p$, we have
$$\nu_p(T(a,b,n))=\sum_{i=1}^{\infty}\left( \left\lfloor \frac{2an}{p^i}\right\rfloor
+\left\lfloor\frac{bn}{p^i}\right\rfloor-\left\lfloor\frac{an}{p^i}\right\rfloor-\left\lfloor\frac{an-bn}
{p^i}\right\rfloor-\left\lfloor\frac{2bn}{p^i}\right\rfloor\right).$$
By Lemma \ref{lem1}, it follows that each term of $\nu_p(T(a,b,n))$ is nonnegative.
Hence $\nu_p(T(a,b,n))\ge 0$. Therefore, $T(a,b,n)\in \mathbb{Z}$.

Since $\gcd(2bn+1,2bn+3)=1$, it suffices to prove that
$$2bn+1|3(a-b)(3a-b)T(a,b,n)$$ and $$2bn+3|3(a-b)(3a-b)T(a,b,n).$$

Now we first prove the latter part.
Suppose that $p^{\alpha}\| 2bn+3$ with $\alpha\ge 1$. Then we shall prove
\begin{equation}\label{eq5}p^{\alpha}| 3(a-b)(3a-b)T(a,b,n).\end{equation}

Let $p^{\beta}\|a-b$ and $p^{\gamma}\|3a-b$
with $\beta\ge 0$ and $\gamma\ge 0$. Write $\tau=\max\{\beta,\gamma\}$.
If $\alpha \le \tau$,
then \eqref{eq5} clearly holds. Now we assume $\alpha > \tau$.

Suppose that $p\ge 5$.
Next we prove
$$\left\lfloor \frac{2an}{p^i}\right\rfloor+\left\lfloor \frac{bn}{p^i}\right\rfloor-\left\lfloor\frac{an}{p^i}\right\rfloor-\left\lfloor\frac{an-bn}
{p^i}\right\rfloor-\left\lfloor\frac{2bn}{p^i}\right\rfloor=1$$
for $i=\tau+1,\tau+2,\ldots,\alpha$. Noting that $p|2bn+3$ and $p\ge 5$, we have $\gcd(p,n)=1$. Otherwise,
we have $p|3$, a contradiction.

By $p^{\alpha}\| 2bn+3$, it follows that $2bn \equiv p^{\alpha}-3 \pmod {p^{\alpha}}$
and $bn\equiv (p^{\alpha}-3)/2 \pmod {p^{\alpha}}$.

Take an integer $i\in \{\tau+1,\tau+2,\ldots,\alpha\}$. Then
$2bn \equiv p^{i}-3 \pmod {p^{i}}$
and $bn\equiv (p^{i}-3)/2 \pmod {p^{i}}$.
Now we divide into several cases according to the value of $an \pmod {p^i}$.

{\bf Case 1.} $an\equiv t \pmod {p^i}$ with $0\le t<(p^i-3)/2$. It follows that
$2an\equiv 2t \pmod {p^i}$ and $0\le 2t< p^i-3$. We also have $$an-bn\equiv t-(p^i-3)/2+p^i
\pmod {p^i},$$ where $0\le t-(p^i-3)/2+p^i<p^i$.
Hence
\begin{eqnarray*}&&\left\lfloor \frac{2an}{p^i}\right\rfloor+\left\lfloor \frac{bn}{p^i}\right\rfloor-\left\lfloor\frac{an}{p^i}\right\rfloor-\left\lfloor\frac{an-bn}
{p^i}\right\rfloor-\left\lfloor\frac{2bn}{p^i}\right\rfloor\\
&=&\frac{2an-2t}{p^i}+\frac{bn-(p^i-3)/2}{p^i}-\frac{an-t}{p^i}\\
&&-
\left(\frac{an-bn-(t-(p^i-3)/2+p^i)}{p^i}\right)-\frac{2bn-(p^i-3)}{p^i}\\
&=&1.
\end{eqnarray*}

{\bf Case 2.} $an\equiv ({p^i}-3)/2 \pmod {p^i}$. Clearly, $an-bn\equiv 0 \pmod {p^i}$. Since
$\gcd(p,n)=1$, we have $p^i|a-b$. However, $p^{\beta}\|a-b$ and $\beta\le \tau<i$, a contradiction.

{\bf Case 3.} $an\equiv ({p^i}-1)/2 \pmod {p^i}$. It follows that
$$3an-bn\equiv \frac {3({p^i}-1)}{2}-\frac{({p^i}-3)}{2}\equiv  0 \pmod {p^i}.$$
By $\gcd(p,n)=1$, we have $p^i|3a-b$. This contradicts $p^{\gamma}\|3a-b$ and $\gamma<i$.

{\bf Case 4.} $an\equiv t \pmod {p^i}$ with $(p^i+1)/2\le t<p^i$. It follows that
$$2an\equiv 2t-p^i \pmod {p^i}, \quad 0\le 2t-p^i< p^i.$$ We also have $$an-bn\equiv t-(p^i-3)/2
\pmod {p^i},\quad 0\le t-(p^i-3)/2<p^i.$$
Hence
\begin{eqnarray*}&&\left\lfloor \frac{2an}{p^i}\right\rfloor+\left\lfloor \frac{bn}{p^i}\right\rfloor-\left\lfloor\frac{an}{p^i}\right\rfloor-\left\lfloor\frac{an-bn}
{p^i}\right\rfloor-\left\lfloor\frac{2bn}{p^i}\right\rfloor\\
&=&\frac{2an-(2t-p^i)}{p^i}+\frac{bn-(p^i-3)/2}{p^i}-\frac{an-t}{p^i}\\
&&-
\left(\frac{an-bn-(t-(p^i-3)/2)}{p^i}\right)-\frac{2bn-(p^i-3)}{p^i}\\
&=&1.
\end{eqnarray*}

Therefore, $\nu_p(T(a,b,n))\ge \alpha-\tau$, and then $$\nu_p(3(a-b)(3a-b)T(a,b,n))\ge \alpha.$$
That is, \eqref{eq5} holds.

Now we deal with the case $p=3$. If $9|n$, then $3|2bn+3$ and $9\nmid 2bn+3$. It follows that $\alpha=1$,
and then \eqref{eq5} clearly holds.
If $9\nmid n$, then we can follow the proof of the case $p\ge 5$ above. In Case 2, by $an-bn\equiv 0 \pmod {3^i}$,
we have $3^{i-1}|a-b$. In Case 3, we have $3^{i-1}|3a-b$. So, if $i\ge \tau+2$, then $i-1\ge \tau+1$. 
It is a contradiction in both cases. Hence
$$\left\lfloor \frac{2an}{3^i}\right\rfloor+\left\lfloor \frac{bn}{3^i}\right\rfloor-\left\lfloor\frac{an}{3^i}\right\rfloor-\left\lfloor\frac{an-bn}
{3^i}\right\rfloor-\left\lfloor\frac{2bn}{3^i}\right\rfloor=1$$
for $i=\tau+2,\tau+3,\ldots,\alpha$.
It follows that $\nu_3(T(a,b,n))\ge \alpha-\tau-1$, and then $$\nu_3(3(a-b)(3a-b)T(a,b,n))\ge \alpha.$$
That is, \eqref{eq5} also holds.
Hence, $2bn+3|3(a-b)(3a-b)T(a,b,n)$.

The proof of $2bn+1|3(a-b)(3a-b)T(a,b,n)$ is very similar. We omit it here.

This completes the proof of Theorem \ref{thm1}.

\end{proof}


\clearpage

\end{document}